\documentclass[reqno]{amsart}
\usepackage{amssymb,amsmath,amsfonts,amsthm,mathrsfs}
\usepackage[colorlinks=true, citecolor=blue, anchorcolor=red]{hyperref}

\newcommand{\R}{\mathbb R}
\newcommand{\N}{\mathbb N}
\newcommand{\C}{\mathbb C}

\newcommand{\E}{\mathbb E}
\newcommand{\F}{\mathbb F}

\renewcommand{\Re}{\mathop{\text{\upshape{Re}}}}
\renewcommand{\Im}{\mathop{\text{\upshape{Im}}}}

\renewcommand{\Im}{\operatorname{Im}}
\newcommand{\ord}{\operatorname{ord}}
\newcommand{\diag}{\operatorname{diag}}

\renewcommand{\epsilon}{\varepsilon}
\renewcommand{\bar}[1]{\overline{#1}}

\renewcommand{\tilde}{\widetilde}

\newtheorem{theorem}{Theorem}[section]
\newtheorem{lemma}[theorem]{Lemma}
\newtheorem{corollary}[theorem]{Corollary}

\theoremstyle{definition}
\newtheorem{definition}[theorem]{Definition}
\newtheorem{remark}[theorem]{Remark}
\newtheorem{example}[theorem]{Example}

\numberwithin{equation}{section}

\allowdisplaybreaks

\begin{document}

\title[Thermoelastic plate equation with free boundary conditions]
{Generation of semigroups for the thermoelastic plate equation with free boundary conditions}
\author{Robert Denk}
\address{Universit\"at Konstanz, Fachbereich f\"ur Mathematik und Statistik,
         78457 Konstanz, Germany}
\email{robert.denk@uni-konstanz.de}
\author{Yoshihiro Shibata}
\address{Department of Mathematical Sciences, School of Science and
Engineering, Waseda University, Ohkobu 3-4-1, Shinjuku-ku, Tokyo 169-8555,
Japan} \email{yshibata@waseda.jp}
\thanks{}
\date{June 26, 2018}

\begin{abstract}
We consider the linear thermoelastic plate equations with free boundary conditions in uniform $C^4$-domains, which includes the half-space, bounded and exterior domains. We show that the corresponding operator generates an analytic semigroup  in $L^p$-spaces for all $p\in(1,\infty)$ and has maximal $L^q$-$L^p$-regularity on finite time intervals. On bounded $C^4$-domains, we obtain exponential stability.
\end{abstract}



\maketitle

\section{Introduction}

Let $\Omega\subset\R^N$ be a domain with boundary $\Gamma$. We consider the linear thermoelastic plate equations
\begin{equation}
  \label{1-1}
  \begin{aligned}
    u_{tt} + \Delta^2 u + \Delta \theta & = f_1\quad \text{ in } (0,\infty)\times \Omega,\\
    \theta_t -\Delta \theta -\Delta u_t & = f_2 \quad \text{ in } (0,\infty)\times \Omega
  \end{aligned}
\end{equation}
with initial conditions
\begin{equation}\label{1-2}
\begin{aligned}
    u|_{t=0} & = u_0\quad\text{ in } \Omega,\\
    u_t|_{t=0} & = u_1\quad\text{ in } \Omega,\\
    \theta|_{t=0} & = \theta_0\quad\text{ in } \Omega.
  \end{aligned}
\end{equation}
System \eqref{1-1} serves as a standard simplified model for thin elastic plates with thermoelastic effects, see \cite{lagnese89}, Chapter~2, or \cite{chuesov-lasiecka10}, for a discussion of this and similar models. In \eqref{1-1},  $u(t,x)$ stands for the vertical displacement at time $t\ge 0$ and at position $x=(x_1,\dots,x_N)\in\Omega$, while $\theta(t,x)$ denotes the temperature (relative to some reference temperature) at time $t$ and position $x$. Note that we omitted all physical constants for simplicity.

Among the physically relevant boundary conditions, the maybe most complicated are the so-called free boundary conditions
\begin{equation}
  \label{1-3}
  \begin{aligned}
    \Delta u - (1-\beta) \Delta' u +\theta & = g_1 \quad \text{ on }(0,\infty)\times\Gamma,\\
    \partial_\nu \big( \Delta u + (1-\beta) \Delta' u + \theta\big) & = g_2 \quad\text{ on }(0,\infty)\times\Gamma,\\
    \partial_\nu\theta & = g_3 \quad\text{ on }(0,\infty)\times \Gamma.
  \end{aligned}
\end{equation}
which will be considered in the present paper. In \eqref{1-3}, $\Delta$ and $\Delta'$ stand for the Laplace operator in $\Omega$ and the Laplace-Beltrami operator on the boundary $\Gamma$, respectively, and $\partial_\nu$ denotes the derivative in outer normal direction. For a survey on other types of boundary conditions and generation of semigroups for them, we refer, e.g., to \cite{lasiecka-triggiani97}. The physically relevant situation is the two-dimensional case $N=2$, but we can consider \eqref{1-1}--\eqref{1-3} in any dimension.

One of the standard approaches to \eqref{1-1}-\eqref{1-3} is to set $v:=\partial_t u$ and obtain the first-order system  acting on $U:=(u,u_t,\theta)^\top$ and being of the form
\begin{equation}
  \label{1-1aa}
U_t - A(D) U = 0 \quad \text{ in } (0,\infty)\times \Omega
\end{equation}
with an operator-matrix $A(D)$ of mixed order. This equation is augmented by boundary conditions of the form
\begin{equation}
\label{1-2aa}
 B(D) U = 0 \quad \text{ on } (0,\infty)\times\Gamma.
\end{equation}
Here $A(D)$ and $B(D)$  are given by
\[ A(D) := \begin{pmatrix}
  0 & 1 & 0 \\
  -\Delta^2 & 0 & -\Delta\\
  0 &  \Delta & \Delta
\end{pmatrix},\quad
B(D) := \begin{pmatrix}
  \Delta - (1-\beta)\Delta' & 0 & 1\\
  \partial_\nu (\Delta + (1-\beta)\Delta') & 0 & 0\\
  0 & 0& \partial_\nu
\end{pmatrix}.\]
The natural space for the  $L^p$-realization of the mixed-order boundary value problem $(A(D),$ $B(D))$ is given by $\E_p^{(0)}(\Omega)$ and its solution space by $\E_p^{(2)}(\Omega)$, where
for $j\in\{0,1,2\}$ we set
\[ \E_p^{(j)}(\Omega) := H_p^{2+j}(\Omega) \times H_p^{j}(\Omega)\times H_p^j(\Omega).\]
More precisely, we define $A_{p,\Omega}$ as an unbounded operator in $\E_p^{(0)}(\Omega)$ with domain
 \[ D(A_{p,\Omega}) := \{ U\in \E_p^{(2)}(\Omega): B(D)U=0\}\]
acting as $A_{p,\Omega} U := A(D)U\;(U\in D(A_{p,\Omega}))$.
We consider uniform $C^4$-domains, see Definition~\ref{1.1} below. The main result of the present paper shows that for all $p\in (1,\infty)$, the operator $A_{p,\Omega}$ generates an analytic $C_0$-semigroup. This is a consequence of the stronger result that $A_{p,\Omega}$ has maximal $L^q$-$L^p$-regularity (Theorem~\ref{3.2}). On bounded $C^4$-domains, we obtain exponential stability (Theorem~\ref{3.6}).

\medskip

The thermoelastic plate equations has been studied by many authors, mostly in an $L^2$-setting. Many results deal with exponential stability of the associated semigroup, e.g.,  \cite{kim92},
 \cite{munos_rivera-racke95},
 \cite{liu-zheng97},
\cite{lasiecka-triggiani97},   \cite{shibata94}.
For the  analyticity of the semigroup, we refer to  \cite{liu-renardy95},
 \cite{liu-liu97}, and   \cite{liu-yong98}
in the $L_2$-setting. For the treatment of nonlinear problems, corresponding results in $L_p$ are of relevance.
In the whole-space case, analyticity of the
generated semigroup in $L_p$ was shown in  \cite{denk-racke06}.
In the case of the half-space and of bounded
domains, equations \eqref{1-1} with Dirichlet (clamped) boundary conditions
\[ u = \partial_\nu u = \theta = 0 \quad\text{ on }
(0,\infty)\times\Gamma\]
were studied in  \cite{naito-shibata09} and \cite{naito09}. In the paper \cite{lasiecka-wilke13}, a rather complete analysis in the $L^p$-setting can be found for hinged boundary conditions $u=\Delta u = \theta =0$.

System \eqref{1-1}--\eqref{1-3}, i.e. the thermoelastic plate equations with free boundary conditions in the $L_p$-setting, has been studied recently by the authors in \cite{Denk-Shibata17}. It was shown that the second-order (in time) system \eqref{1-1}--\eqref{1-3} has maximal $L_q$-$L_p$-regularity. However, this does not imply that the first-order system \eqref{1-1aa}--\eqref{1-2aa} generates an analytic $C_0$-semigroup. This was also observed in the case of the structurally damped plate equation with clamped boundary conditions in \cite{denk-schnaubelt15}. In fact, in the situation of \cite{denk-schnaubelt15}, we have maximal regularity, but no generation of semigroup unless additional conditions are included in the basic space. Roughly speaking, this is due to the fact that the standard resolvent estimates hold only for right-hand sides with vanishing first component, and the reformulation of \eqref{1-1} as a first-order system in fact leads to such a right-hand side.

In the present paper, however, we show that the operator related to the first-order system \eqref{1-1a}  generates an analytic $C_0$-semigroup without additional conditions on the basic space $\E_p^{(0)}(\Omega)$. The proofs are based on Fourier multiplier methods on one hand and  on the results from \cite{Denk-Shibata17} on the other hand. If the domain $\Omega$ is bounded, we obtain exponential stability apart from the kernel of the operator. In particular, we obtain generation of an analytic semigroup and exponential stability for the two-dimensional system which was studied in \cite{lasiecka-triggiani98c}, in this way generalizing the results in \cite{lasiecka-triggiani98c} from the $L^2$-case to the $L^p$-case.

\section{The whole space case}

In this section, we consider the whole-space case, i.e. system \eqref{1-1}--\eqref{1-2} with $\Omega=\R^N$. Our approach is based on the Fourier transform and results on vector-valued Fourier multipliers. In particular, the proof of maximal regularity in the sense of well-posedness in $L^q$-$L^p$-Sobolev spaces make use of the concept of $\mathcal R$-boundedness and variants of Michlin's theorem. As standard references, we mention \cite{denk-hieber-pruess03} and \cite{kunstmann-weis04}.

The Fourier transform $\mathscr F$ in $\R^N$ is given by
\[ (\mathscr F\varphi)(\xi) := (2\pi)^{-N/2} \int_{\R^N} \varphi(x)e^{-ix\xi} dx\quad (\xi\in\R^N)\]
for Schwartz functions $\varphi$ and extended by duality to tempered distributions. A symbol $m\in L^\infty(\R^N)$ is called a Fourier multiplier if $\mathscr F^{-1} m\mathscr F$ defines a bounded linear operator in $L^p(\R^N)$.
One of the key ingredients to show $\mathcal R$-sectoriality will be the vector-valued version of Michlin's theorem on Fourier multipliers due to Weis \cite{weis01} and Girardi and Weis \cite{girardi-weis03}.

The following definition is a variant of \cite{Denk-Shibata17}, Definition~3.2.

\begin{definition}
  \label{2.1}
  Let $\Sigma\subset \C$ be a set, let $m\colon (\R^{N}\setminus\{0\})\times \Sigma\to \C, \, (\xi,\lambda)\mapsto m(\xi,\lambda)$, be $C^\infty$ with respect to $\xi$. Let $s\in\R$.  Then $m$ is called a multiplier of order $s$ in $\Sigma$ if the estimates
  \[ |\partial_{\xi}^\alpha m(\xi,\lambda)| \le C_\alpha (|\lambda|^{1/2} + |\xi|)^{s} |\xi|^{-|\alpha|}\]
  hold for any multi-index $\alpha\in\N_0^N$ and $(\xi,\lambda)\in (\R^{N}\setminus\{0\})\times \Sigma$ with some constant $C_\alpha$ depending only on $\alpha$ and $\Sigma$. The set of all multipliers of order $s$  in $\Sigma$ will be denoted by $\mathbb M_s(\Sigma)$.
\end{definition}

It is easily seen that $\mathbb M_s(\Sigma)$ is a complex vector space and that  for $m_1\in \mathbb M_{s_1}(\Sigma)$ and $m_2\in \mathbb M_{s_2}(\Sigma)$ we have $m_1m_2\in \mathbb M_{s_1+s_2}(\Sigma)$ (see \cite{Denk-Shibata17}, Lemma~3.3).

\begin{example}
  \label{2.2} We mention some examples which will be useful below. Let $\theta\in (0,\pi)$, and let
\begin{equation}\label{2-1a}
 \Sigma_\theta := \{ \lambda\in\C\setminus\{0\}: |\arg\lambda|<\theta\}
\end{equation}
be the open sector in the complex plane.

a) Directly from the definition it can be seen that $ \lambda\in \mathbb M_2(\Sigma_\theta)$ (where $\lambda$ stands for the constant mapping $(\xi,\lambda)\mapsto \lambda$) and $|\xi|^{2k}\in M_{2k}(\Sigma_\theta)$ for all $k\in\N$.

b) Let $s\in\R\setminus\{0\}$, and let $m(\xi,\lambda):= (\lambda+|\xi|^2)^{s/2}$. Then $m\in \mathbb M_s(\Sigma_\theta)$. This can be seen by homogeneity: As $m$ is quasi-homogeneous of order $s$ in the sense that
\[ m(\rho\xi,\rho^2 \lambda) = \rho^s m(\xi,\lambda)\quad (\rho>0,\, \xi\in\R^N\setminus\{0\},\, \lambda\in\Sigma_\theta),\]
the derivative $\partial_\xi^\alpha m$ is quasi-homogeneous of degree $s-|\alpha|$. By a compactness argument,
\[ |\partial_\xi^\alpha m(\xi,\lambda)| \le c_\alpha (|\lambda|^{1/2} + |\xi|)^{s-|\alpha|} \le c_\alpha (|\lambda|^{1/2} + |\xi|)^s |\xi|^{-|\alpha|}\]
which shows $m\in\mathbb M_s(\Sigma_\theta)$.

c) By a similar homogeneity argument, we see that $(\xi,\lambda)\mapsto \frac{|\xi|}{(1+|\xi|^2)^{1/2}} \in \mathbb M_0(\Sigma_\theta)$.

d) Let $\lambda_0>0$. Then $(\xi,\lambda)\mapsto 1\in \mathbb M_2(\lambda_0+\Sigma_0)$ due to
\[ 1\le c_{\lambda_0} |\lambda| \le c_{\lambda_0} (|\lambda|^{1/2}+|\xi|)^2.\]
Therefore, $1+|\xi|^2\in \mathbb M_2(\lambda_0+\Sigma_\theta)$. Note that $1\not\in \mathbb M_2(\Sigma_\theta)$.
\end{example}

The following result is one main tool for the results below and was shown in \cite{enomoto-shibata13}, Theorem~3.3.

\begin{lemma}
  \label{2.3}
  In the situation of Definition~\ref{2.1}, let $m\in \mathbb M_0(\Sigma)$. For $\lambda\in\Sigma$, define the operator $m(D,\lambda)$ by $m(D,\lambda) f= \mathscr F_\xi^{-1}[m(\lambda,\xi)\mathscr F_\xi f(\xi)]$. Then, the family $\{m(D,\lambda):\lambda\in\Sigma\}\subset L(L^p(\R^N))$ is $\mathcal R$-bounded and
  \[ \mathcal R_{L(L^p(\R^N))}(\{m(D,\lambda):\lambda\in\Sigma\}) \le C_{p,N} \max_{|\alpha|\le N+1} C_\alpha\]
  with $C_{p,N}$ depending only on $p$ and $N$.
\end{lemma}

The analysis of the operator $A_{p,\R^N}$ in the whole space was essentially done in \cite{naito-shibata09} and \cite{naito09}. We summarize some results from these papers. Define $\gamma_1,\gamma_2,\gamma_3$ by the equality
\begin{equation}
\label{1-1a}
 p(t) := t^3+t^2+2t+1 = (t+\gamma_1)(t+\gamma_2)(t+\gamma_3)
\end{equation}
with $\gamma_1\in\R$, $\gamma_2=\bar\gamma_3$ and $ \Im \gamma_2>0$. Then $\gamma_1\in (0,1)$, $\Re\gamma_2=\Re\gamma_3\in (0,\frac12)$, and $\det(\lambda-A(\xi)) = \prod_{j=1}^3 (\lambda+\gamma_j|\xi|^2)$ (see \cite{naito-shibata09}, Lemma~2.3). We define $\vartheta_0:= \arg(-\gamma_3)\in (\frac\pi 2,\pi)$ (note that $-\gamma_3$ is the root of the polynomial $p$ with positive imaginary part).

We consider the whole space resolvent
\begin{equation}\label{1-4}
 R(\lambda) := (\lambda-A(D))^{-1} := \mathscr F_{\xi}^{-1} (\lambda-A(\xi))^{-1}\mathscr F
\end{equation}
with
\[ A(\xi) := \begin{pmatrix}
  0 & 1 & 0 \\
  -|\xi|^4 & 0 & |\xi|^2\\
  0 & -|\xi|^2 & -|\xi|^2
\end{pmatrix}.\]
For $j\in\{0,1,2\}$, define
\[ S_j(\xi) := (1+|\xi|^2)^{j/2}\diag((1+|\xi|^2), 1,1),\]
 where $\diag(\ldots)$ stands for the diagonal matrix with the corresponding elements on the diagonal.
For the next result, we use the fact that the induced operator
$S_j(D)$ defines an  isometric isomorphism
\begin{equation}
\label{2-1}
  S_j(D)\in L_{\text{Isom}}( \E_p^{(j)} , L^p(\R^N;\C^3)) \quad (j\in\{0,1,2\}).
\end{equation}

\begin{lemma}
  \label{2.4}
  For every $\vartheta<\vartheta_0$, $\lambda_0>0$ and $j\in\{0,1,2\}$ we have
  \[ \mathcal R_{L(\E_p^{(0)}(\R^N), \E_p^{(2-j)}(\R^N))}\big( \big\{ \lambda^{j/2} R(\lambda): \lambda\in\lambda_0+\Sigma_{\vartheta}\big\}\big) <\infty.\]
\end{lemma}

\begin{proof}
  Let $j\in\{0,1,2\}$. In view of \eqref{2-1} and Lemma~\ref{2.2}, we have to show that every entry of the matrix
  \[ M^{(j)}(\xi,\lambda) = \big( m^{(j)}_{kl}(\xi,\lambda)\big)_{k,l=1,2,3} := \lambda^{j/2} S_{2-j}(\xi)(\lambda-A(\xi))^{-1} S_0(\xi)^{-1}\]
  belongs to $\mathbb M_0(\lambda_0+\Sigma_{\vartheta})$.
  It was shown in \cite{naito-shibata09}, Section~2, that for all $\lambda\in\lambda_0+\Sigma_{\vartheta}$ we have
  \[ (\lambda-A(\xi))^{-1} = \frac1{\det(\lambda-A(\xi))}
  \begin{pmatrix}
    \lambda (\lambda+|\xi|^2) + |\xi|^4 & \lambda+|\xi|^2 & |\xi|^2\\
    -(\lambda+|\xi|^2)|\xi|^4 & \lambda(\lambda+|\xi|^2) & \lambda|\xi|^2\\
    |\xi|^6 & -\lambda|\xi|^2 & \lambda^2+|\xi|^4
  \end{pmatrix}\,.\]
Moreover, $\det(\lambda-A(\xi)) = \prod_{i=1}^3 (\frac\lambda{\gamma_i}+|\xi|^2)$ (\cite{naito-shibata09}, Lemma~2.3). As in Example~\ref{2.2} b), we see that
\[ (\xi,\lambda)\mapsto (\tfrac\lambda{\gamma_i}+|\xi|^2)^{-1}\in \mathbb M_{-2}(\Sigma_\vartheta)\]
and therefore $(\det(\lambda-A(\xi)))^{-1}\in\mathbb M_{-6}(\Sigma_\vartheta)$. For the left upper corner of $M^{(j)}$, we have
\[ m^{(j)}_{11}(\xi,\lambda) = (\det(\lambda-A(\xi)))^{-1} \lambda^{j/2} (1+|\xi|^2)^{(2-j)/2} (\lambda^2+\lambda|\xi|^2+|\xi|^4).\]
With Example~\ref{2.2} we see that
\begin{align*}
  \lambda^{j/2}& \in\mathbb M_j(\Sigma_\vartheta),\\
  (1+|\xi|^2)^{(2-j)/2} & \in \mathbb M_{2-j}(\lambda_0+\Sigma_\vartheta),\\
  \lambda^2+\lambda|\xi|^2+|\xi|^4 & \in \mathbb M_4(\Sigma_\vartheta),
\end{align*}
which yields $m_{11}^{(j)}\in \mathbb M_0(\lambda_0+\Sigma_\vartheta)$. Similarly,
\[ m_{21}^{(j)}(\xi,\lambda) = - (\det(\lambda-A(\xi)))^{-1} \lambda^{j/2} (1+|\xi|^2)^{(2-j)/2} (\lambda+|\xi|^2) \tfrac{|\xi|^2}{1+|\xi|^2}.\]
Using
\begin{align*}
  \lambda^{j/2} & \in\mathbb M_j(\Sigma_\vartheta),\\
  (1+|\xi|^2)^{(2-j)/2} & \in \mathbb M_{2-j}(\lambda_0+\Sigma_\vartheta),\\
  (\lambda+|\xi|^2) & \in \mathbb M_2(\Sigma_\vartheta),\\
  \tfrac{|\xi|^2}{1+|\xi|^2} & \in \mathbb M_0(\Sigma_\vartheta),
\end{align*}
we obtain $m_{21}\in\mathbb M_0(\lambda_0+\Sigma_\vartheta)$. All other entries of the matrix $M^{(j)}$ can be estimated similarly. Therefore, $M^{(j)}\in \mathbb M_0(\lambda_0+\Sigma_\vartheta)$ which finishes the proof.
 \end{proof}

\begin{corollary}
  \label{2.5} a) For all $\lambda\in\Sigma_{\vartheta_0}$, the operator $\lambda - A_{p,\R^N}\colon \E_p\to \F_p$ is invertible.

  b) The operator $A_{p,\R^N}$ is not sectorial for any angle and therefore does not generate a bounded $C_0$-semigroup on $\F_p$.

  c) For any $\lambda_0>0$, the operator $A_{p,\R^N}-\lambda_0$ is $\mathcal R$-sectorial with $\mathcal R$-angle $\vartheta_0$. Therefore, $A_{p,\R^N}-\lambda_0$ has maximal $L^q$-$L^p$-regularity in $(0,\infty)$, and $A_{p,\R^N}$ has maximal $L^q$-$L^p$-regularity in $(0,T)$ with $T<\infty$. In particular, $A_{p,\R^N}$ generates an analytic $C_0$-semigroup.
\end{corollary}

\begin{proof}
a) Let $\lambda\in\Sigma_{\vartheta_0}$ and choose $\vartheta<\vartheta_0$ and $\lambda_0>0$ such that $\lambda\in \lambda_0+\Sigma_\vartheta$. By Lemma~\ref{2.4} with $j=0$, we have $R(\lambda)\in L(\E_p^{(0)}, \E_p^{(2)})$. Obviously, $R(\lambda)$ is the inverse of $\lambda-A_{p,\R^N}$, and therefore $\lambda$ is in the resolvent set of $A_{p,\R^N}$.

b) Assume that $\|\lambda(\lambda-A_{p,\R^N})^{-1} \|_{L(\F_p)} \le C\;(\lambda\in (0,\infty))$ holds. Then the operator $M_0(D,\lambda)\in L(L^p(\R^N;\C^3))$ is uniformly (with respect to $\lambda$) bounded, where $M_0(\xi,\lambda) := \lambda S_1(\xi) (\lambda-A(\xi))^{-1} S_1(\xi)^{-1}$. In particular, every entry of $M_0(\xi,\lambda)$ is an $L^\infty$-function (see Prop.~3.17 in \cite{denk-hieber-pruess03}). For the last entry in the first row of $M_0(\xi,\lambda)$, we obtain
\[ \left| \frac{\lambda(1+|\xi|^2)|\xi|^2}{\prod_{j=1}^3 (\lambda+\gamma_j|\xi|^2)}\right| \le C<\infty \quad (\lambda\in (0,\infty),\, \xi\in\R^N).\]
However, setting $\lambda=k^{-2}$ and $|\xi|=k^{-1}$, we see that the left-hand side is unbounded for $k\to\infty$.

c) The $\mathcal R$-sectoriality  follows from Lemma~\ref{2.4} with $j=2$, and the other statements are consequences of the general theory on $\mathcal R$-sectorial operators.
\end{proof}

\section{The case of a uniform $C^4$-domain}

\begin{definition}
  \label{1.1}
A domain $\Omega$ is called a uniform $C^4$-domain if
there exist positive constants $\alpha$, $\beta$ and $K$
such that for any $x_0 \in \Gamma$ there exist a coordinate
number $j$ and a $C^4$-function
$h(x')$
defined on $B'_\alpha(x'_0)$ such that
$\|h\|_{H^4_\infty(B'_\alpha(x_0'))} \leq K$ and
\begin{align*}
\Omega\cap B_\beta(x_0) & = \{x \in \R^N \mid
x_j > h(x') \enskip(x' \in B'_\alpha(x'_0))\}
\cap B_\beta(x_0), \\
\Gamma\cap B_\beta(x_0) & = \{x \in \R^N \mid
x_j = h(x') \enskip(x' \in B'_\alpha(x'_0))\}
\cap B_\beta(x_0).
\end{align*}
Here, $x'$ has been defined by
$x' = (x_1, \ldots, x_{j-1}, x_{j+1}, \ldots, x_{_N})$
for $x=(x_1,\ldots, x_{_N})$,
$$B'_\alpha(x'_0)
= \{x' \in \R^{N-1} \mid |x'-x'_0| < \alpha\},
\quad B_\beta(x_0) = \{x \in \R^N \mid |x-x_0| < \beta\}.
$$
\end{definition}

Let $\Omega\subset\R^N$ be a uniform $C^4$-domain with boundary $\Gamma$, and let $p\in(1,\infty)$. To show that the operator $A_p\colon \E_p^{(0)}(\Omega)\supset D(A_p)\to \E_p^{(0)}(\Omega)$ generates an analytic $C_0$-semigroup, we first consider the boundary value problem
\begin{equation}
  \label{3-1}
  \begin{aligned}
    (\lambda-A(D))U & = 0\quad\text{ in }\Omega,\\
    B(D)U & = G\quad\text{ on }\Gamma.
  \end{aligned}
\end{equation}
Here, $G=(g_1,g_2,g_3)^\top$ is defined in the whole of $\Omega$. Similarly to \cite{Denk-Shibata17}, (1.8), we define the spaces
\begin{align*}
  \mathbb G_p(\Omega) & := H_p^2(\Omega)\times H_p^1(\Omega)\times H_p^1(\Omega),\\
  \mathbb X_p(\Omega) & := \mathbb G_p(\Omega)\times\big( H_p^1(\Omega)\times L^p(\Omega)\times L^p(\Omega)\big)\times L^p(\Omega).
\end{align*}
The $\lambda$-dependent map $H(\lambda)\colon \mathbb G_p(\Omega)\to \mathbb X_p(\Omega)$ is defined by
\[ H(\lambda)((g_1,g_2,g_3)^\top) := \big((g_1,g_2,g_3),\lambda^{1/2} (g_1,g_2,g_3), \lambda g_1\big)^\top.\]
The following result on the existence of $\mathcal R$-bounded solution operators was shown in \cite{Denk-Shibata17}, Theorem~1.4.

\begin{theorem}
  \label{3.1}
  There exist a number $\vartheta_1\in (\frac\pi2,\pi)$, a positive number $\lambda_0$ and an operator family
  \[ L(\lambda) = \big(L_1(\lambda),\lambda L_1(\lambda), L_2(\lambda)\big)^\top \in L(\mathbb X_p(\Omega), \E_p^{(2)}(\Omega))\]
  such that for every $\lambda\in\lambda_0+\Sigma_{\vartheta_1}$ and every $G\in\mathbb G_p(\Omega)$, problem \eqref{3-1} admits a unique solution $U\in \E_p^{(2)}(\Omega)$ given by $U = L(\lambda) H(\lambda)G$. Moreover,
  \begin{align*}
  \mathcal R_{L(\mathbb X_p(\Omega), H_p^{4-j}(\Omega))} \big(\big\{
  \lambda^{j/2} L_1(\lambda)\colon \lambda\in \lambda_0+\Sigma_{\vartheta_1}\big\} \big) \le C\quad (j=0,\dots,4),\\
  \mathcal R_{L(\mathbb X_p(\Omega), H_p^{2-j}(\Omega))} \big(\big\{
  \lambda^{j/2} L_2(\lambda)\colon \lambda\in \lambda_0+\Sigma_{\vartheta_1}\big\} \big) \le C\quad (j=0,1,2).
  \end{align*}
\end{theorem}

Let $R_\Omega\colon f\mapsto f|_\Omega$ denote the restriction of a function defined on $\R^N$ to $\Omega$. Obviously, we have $r_\Omega\in L(H_p^i(\R^N), H_p^i(\Omega))$ with norm 1 for any $i\in\N_0$. In fact, $r_\Omega$ is a retraction as a corresponding co-retraction (extension operator) exists for uniform $C^4$-domains. In the following, we fix an extension operator $e_\Omega\colon L^1_{\text{loc}}(\Omega)\to L^1_{\text{loc}}(\R^N)$ with the property that for any $p\in (1,\infty)$ and $f\in H_p^i(\Omega)$, we have $e_\Omega f\in H_p^i(\R^N)$, $r_\Omega e_\Omega f=f$, and $ \|e_\Omega\|_{L(H_p^i(\Omega), H_p^i(\R^N))} \le C_p$ for $i=0,\dots,4$. For the existence of such an extension operator, we refer to \cite{schade-shibata15}, Appendix~A.

The following theorem is the main result of the present paper.

\begin{theorem}
  \label{3.2}
  There exist $\lambda_0>0$ and $\vartheta>\frac\pi 2$ such that the operator $A_{p,\Omega}-\lambda_0$ is $\mathcal R$-sectorial with $\mathcal R$-angle $\vartheta$. Therefore, $A_{p,\Omega}$ has maximal $L^q$-$L^p$-regularity in every finite time interval. In particular, $A_{p,\Omega}$ generates an analytic $C_0$-semigroup in $\E_p^{(0)}(\Omega)$.
\end{theorem}

\begin{proof}
  We first obtain a description of the resolvent $(\lambda-A_{p,\Omega})^{-1}$. For this, let $F\in \E_p^{(0)}(\Omega)$ be given. We apply the extension operator $e_\Omega$ from above to every component of $F$ and obtain $e_\Omega F\in \E_p^{(0)}(\R^N)$. We set $U_1 := r_\Omega R(\lambda)e_\Omega F$ for $\lambda\in\Sigma_{\vartheta_0}$ with $R(\lambda)$ being the whole space resolvent defined in \eqref{1-4}.

  To solve
  \begin{equation}
  \label{3-2}
  \begin{aligned}
    (\lambda - A(D)) U & = F \quad\text{ in } \Omega,\\
    B(D) U & = 0 \quad\text{ on }\Gamma,
  \end{aligned}
  \end{equation}
  we set $U=U_1+U_2$ and obtain the boundary value problem
  \begin{align*}
    (\lambda - A(D)) U_2 & = 0 \quad \text{ in }\Omega,\\
    B(D) U_2 & = - B(D) U_1 \quad\text{ on }\Gamma
  \end{align*}
  for $U_2$. Due to Theorem~\ref{3.1}, there exist $\lambda_0$ and $\vartheta_1$ such this equation is uniquely solvable for $\lambda\in \lambda_0+\Sigma_{\vartheta_1}$, and its solution is given by
  \[ U_2 = - L(\lambda) H(\lambda) B(D) U_1.\]
  Therefore, for $\vartheta\in (\frac\pi 2, \min\{\vartheta_0,\vartheta_1\})$ and $\lambda\in \lambda_0+\Sigma_\vartheta$, the boundary value problem \eqref{3-2} is uniquely solvable with solution
  \[ U = U_1+U_2 = r_\Omega R(\lambda) e_\Omega F - L(\lambda)H(\lambda)B(D)r_\Omega R(\lambda)e_\Omega F.\]
  Consequently, we have to show the $\mathcal R$-boundedness of the operator family
  \begin{equation}
    \label{3-3}
    \lambda(\lambda-A_{p,\Omega})^{-1} = r_\Omega \lambda R(\lambda) e_\Omega - \lambda L(\lambda) H(\lambda) B(D) r_\Omega R(\lambda) e_\Omega\quad (\lambda\in\lambda_0+\Sigma_\vartheta).
  \end{equation}
  By Corollary~\ref{2.5} c), $\lambda R(\lambda)$ is $\mathcal R$-bounded. As $e_\Omega$ and $r_\Omega$ are continuous and $\lambda$-independent, we obtain
  \begin{equation}
    \label{3-4}
    \mathcal R_{L(\E_p^{(0)}(\Omega))} \big(\big\{ \lambda r_\Omega R(\lambda) e_\Omega\colon \lambda\in \lambda_0+\Sigma_\vartheta\big\}\big) <\infty.
  \end{equation}
  Similarly, by Theorem~\ref{3.1} with $j=2$ and $j=4$ we see
  \begin{equation}
    \label{3-5}
    \mathcal R_{L(\mathbb X_p(\Omega), \E_p^{(0)}(\Omega))} \big(\big\{ \lambda L(\lambda)\colon \lambda\in \lambda_0+\Sigma_\vartheta\big\}\big) <\infty.
  \end{equation}
  It remains to show that the family $H(\lambda)B(D)r_\Omega R(\Lambda) e_\Omega$ is $\mathcal R$-bounded. By the definition of the matrix $B(D)$ and the spaces, we see that the operators
  \begin{align*}
    B(D)&\colon \E_p^{(2)}(\Omega)\to \mathbb G_p(\Omega),\\
    B(D)&\colon \E_p^{(1)}(\Omega) \to H_p^1(\Omega)\times L^p(\Omega)\times L^p(\Omega),\\
    B_1(D)&\colon \E_p^{(0)}(\Omega)\to L^p(\Omega)
  \end{align*}
  are continuous, where $B_1(D)$ stands for the first row of $B(D)$, i.e. $B_1(D)(u,v,\theta)^\top = (\Delta-(1-\beta)\Delta') u + \theta$. Thus,
  \begin{equation}
    \label{3-6}
    \begin{pmatrix}
      B(D) \\ B(D)\\ B_1(D)
    \end{pmatrix}
    \colon \E_p^{(2)}(\Omega) \times \E_p^{(1)}(\Omega)\times \E_p^{(0)}(\Omega) \to \mathbb X_p(\Omega)
  \end{equation}
  is continuous (and independent of $\lambda$). By Corollary~\ref{2.5} c), the family
  \[ \left\{ \begin{pmatrix}
    R(\lambda)\\ \lambda^{1/2} R(\lambda)\\ \lambda R(\lambda)
  \end{pmatrix}
  \colon \lambda\in\lambda_0+\Sigma_\vartheta\right\}\subset L\big( \E_p^{(0)}(\Omega), \E_p^{(2)}(\Omega) \times \E_p^{(1)}(\Omega)\times \E_p^{(0)}(\Omega)\big)\]
  is $\mathcal R$-bounded. In combination with
  \[ H(\lambda)B(D)r_\Omega R(\lambda) e_\Omega F =\begin{pmatrix}
    B(D) r_\Omega R(\lambda)e_\Omega F\\
   B(D) r_\Omega \lambda^{1/2} R(\lambda)e_\Omega F\\
    B_1(D) r_\Omega \lambda R(\lambda)e_\Omega F
  \end{pmatrix}\]
  and \eqref{3-6}, this yields
  \begin{equation}
    \label{3-7}
    \mathcal R_{L(\E_p^{(0)}(\Omega), \mathbb X_p(\Omega))} \big\{ H(\lambda) B(D) r_\Omega R(\lambda) e_\Omega: \lambda\in \lambda_0+\Sigma_\vartheta\big\} <\infty.
  \end{equation}
  From \eqref{3-4}, \eqref{3-5}, and \eqref{3-7}, the first statement of the theorem follows by the description of the resolvent in \eqref{3-3}. As before, the other statements follow by the general theory of $\mathcal R$-boundedness.
\end{proof}

The results of Theorem~\ref{3.2} are preserved under lower-order perturbations of the operators $A(D)$ and $B(D)$. More precisely, we consider perturbation matrices of the form
\begin{align*}
  A'(D) & = \begin{pmatrix}
    0 & 0 & 0 \\
    a_{21}(x,D) & 0 & a_{23}(x,D)\\
    0 & a_{32}(x,D) & a_{33}(x,D)
  \end{pmatrix},\\
  B'(D) & = \begin{pmatrix}
    b_{11}(x,D) & 0 & 0 \\
    b_{21}(x,D) & 0 & 0 \\
    0 & 0 & b_{33}(x,D)
  \end{pmatrix}.
  \end{align*}
Here $a_{ij}(x,D)$ and $b_{ij}(x,D)$ are linear differential operators. With respect to the orders of the operators, we assume $\ord a_{ij}(x,D)\le s_{ij}$ and $\ord b_{ij}(x,D)\le t_{ij}$ with $s_{21}=3,\, s_{23}=s_{32}=s_{33}=1$ and $t_{11}=1,\, t_{21}=2,\, t_{33}=0$. The coefficients of $a_{ij}(x,D)$ are assumed to belong to $H_\infty^{s_{ij}-1}(\Omega)$, while the coefficients of $b_{11}(x,D),\, b_{21}(x,D)$, and $b_{33}(x,D)$ are assumed to belong to $H_\infty^2(\Omega)$, $H_\infty^1(\Omega)$, and $H_\infty^1(\Omega)$, respectively.

\begin{lemma}
  \label{3.3} Let $(A'(D),B'(D))$ be a lower-order perturbation as described above. Define the perturbed operator $\tilde A_{p,\Omega}\colon \E_p^{(0)}\supset D(\tilde A_{p,\Omega})\to \E_p^{(0)}(\Omega)$ by
  \[ D(\tilde A_{p,\Omega}) := \{ U\in \E_p^{(2)}(\Omega): \tilde B(D) U =0 \},\; \tilde A_{p,\Omega}U := \tilde A(D) U,\]
  where $\tilde A(D) := A(D) + A'(D)$ and $\tilde B(D) := B(D) + B'(D)$.

  Then there exist $\lambda_0>0$ and $\vartheta>\frac\pi 2$ such that the operator $\tilde A_{p,\Omega}-\lambda_0$ is $\mathcal R$-sectorial with $\mathcal R$-angle $\vartheta$. In particular, $\tilde A_{p,\Omega}$ has maximal $L^q$-$L^p$-regularity in every finite time interval and generates an analytic $C_0$-semigroup in $\E_p^{(0)}(\Omega)$.
\end{lemma}

\begin{proof}
  \textbf{(i)} First, we consider boundary perturbations, i.e. $A'(D)=0$. As in the proof of Theorem~\ref{3.2}, we have to find a solution $\tilde U_2$ of the boundary value problem
  \begin{equation}
    \label{3-8}
    \begin{aligned}
      (\lambda - A(D)) \tilde U_2 & = 0 \quad \text{ in }\Omega,\\
      \tilde B(D) \tilde U_2 & = G\quad\text{ on }\Gamma,
    \end{aligned}
  \end{equation}
  where $G:=-\tilde B(D) r_\Omega R(\lambda) e_\Omega F$. We set $U:=L(\lambda) H(\lambda) G$. Then $U$ solves
  \begin{align*}
    (\lambda-A(D)) U & = 0\quad\text{ in }\Omega,\\
    \tilde B(D) U & = G-B'(D) U = (1-B'(D)L(\lambda)H(\lambda)) G =: \tilde G \quad\text{ on }\Gamma.
  \end{align*}
  Let $\lambda_0$ and $\vartheta$ be as in the proof of Theorem~3.2. We show that the operator family
  \[ \big\{ \lambda^{1/2} H(\lambda) B'(D) L(\lambda):\lambda\in \lambda_0+\Sigma_\vartheta\big\}\subset L(\mathbb X_p(\Omega))\]
  is $\mathcal R$-bounded. In fact, due to the assumptions on $B'(D)$, we have
  \begin{align*}
    b_{11}(x,D) & \in L(H_p^{4-j}(\Omega),H_p^{3-j}(\Omega))\quad (j=1,2,3),\\
    b_{21}(x,D) & \in L(H_p^{4-j}(\Omega),H_p^{2-j}(\Omega))\quad(j=1,2),\\
    b_{33}(x,D) & \in L(H_p^{2-j}(\Omega), H_p^{2-j}(\Omega))\quad(j=1,2).
  \end{align*}
  By Theorem~\ref{3.1}, the families
  \begin{align*}
    \big\{ \lambda^{j/2} L_1(\lambda):\lambda\in\lambda_0+\Sigma_\vartheta\big\} & \subset L(\mathbb X_p(\Omega), H_p^{4-j}(\Omega))\quad (j=0,\ldots,4),\\
    \big\{ \lambda^{j/2} L_2(\lambda):\lambda\in\lambda_0+\Sigma_\vartheta\big\} & \subset L(\mathbb X_p(\Omega), H_p^{2-j}(\Omega))\quad (j=0,1,2)
  \end{align*}
  are $\mathcal R$-bounded. By composition, we see that the family
  \[ \left\{ \lambda^{1/2} H(\lambda) B'(D)L(\lambda) = \begin{pmatrix}
    \lambda^{1/2} B'(D) L(\lambda)\\
    \lambda B'(D) L(\lambda)\\
    \lambda^{3/2} b_{11}(x,D)L_1(\lambda)
  \end{pmatrix}: \lambda\in \lambda_0+\Sigma_\vartheta\right\}\subset L(\mathbb X_p(\Omega))\]
  is $\mathcal R$-bounded. Choosing $\lambda_1>\lambda_0$ sufficiently large, we obtain
  \begin{equation}
    \label{3-9}
    \mathcal R_{L(\mathbb X_p(\Omega))} \Big( \big\{H(\lambda)B'(D)L(\lambda)\colon \lambda\in\lambda_1+\Sigma_\vartheta\big\}\Big) \le \frac 12.
  \end{equation}
  Therefore, $1-H(\lambda)B'(D)L(\lambda)\in L(\mathbb X_p(\Omega))$ is invertible for all $\lambda\in \lambda_1+\Sigma_\vartheta$. A simple algebraic calculation shows that this implies that also $1-B'(D)L(\lambda)H(\lambda)\in L(\mathbb G_p(\Omega))$ is invertible, and that we have
  \begin{equation}
    \label{3-10}
    H(\lambda) (1-B'(D)L(\lambda)H(\lambda))^{-1} = (1-H(\lambda)B'(D)L(\lambda))^{-1} H(\lambda).
  \end{equation}
  Setting $\tilde U_2 := L(\lambda)H(\lambda)(1-B'(D)L(\lambda)H(\lambda))^{-1} G$, we obtain $(\lambda-A(D))\tilde U_2=0$ and
  \[ \tilde B(D)\tilde U_2 = (1-B'(D)L(\lambda)H(\lambda)) (1-B'(D)L(\lambda)H(\lambda))^{-1} G = G,\]
  i.e., $\tilde U_2$ is a solution of \eqref{3-8}. As in the proof of Theorem~\ref{3.2}, the solution $\tilde U$ of the resolvent equation is now given by $\tilde U = U_1 + \tilde U_2$ with $U_1 := r_\Omega R(\lambda) e_\Omega$ as in the proof of Theorem~\ref{3.2}. Therefore, the resolvent of $\tilde A_{p,\Omega}$ is given by
  \begin{align*}
    (\lambda-\tilde A_{p,\Omega})^{-1} & = r_\Omega R(\lambda) e_\Omega - L(\lambda) H(\lambda) (1-B'(D)L(\lambda)H(\lambda))^{-1}\tilde B(D) r_\Omega R(\lambda) e_\Omega\\
    & = r_\Omega R(\lambda) e_\Omega - L(\lambda) (1-H(\lambda) B'(D)L(\lambda))^{-1} H(\lambda) \tilde B(D) r_\Omega R(\lambda)e_\Omega,
  \end{align*}
  where we used \eqref{3-10} for the last equality. We have already seen in the proof of Theorem~\ref{3.2} that the operator families $\lambda L(\lambda)$ and $H(\lambda)\tilde B(D) r_\Omega R(\lambda) e_\Omega$ are $\mathcal R$-bounded. Using \eqref{3-9} and a Neumann series argument, we see that
  \[ \mathcal R_{L(\mathbb X_p)} \big( \big\{ (1-H(\lambda)B'(D)L(\lambda))^{-1}: \lambda\in\lambda_1+\Sigma_\vartheta\big\}\big) \le 2.\]
  Now the statements of the lemma follow in the same way as in the proof of Theorem~\ref{3.2}.

  \medskip

  \textbf{(ii)} In the case $A'(D)\not=0$, we consider $(\tilde A(D), \tilde B(D))$ as a perturbation of $(A(D),\tilde B(D))$. Let $\tilde A_{\tilde B}$ and $A_{\tilde B}$ denote the corresponding operators, respectively. Note that we have $D(\tilde A_{\tilde B})=D(A_{\tilde B})$. By the interpolation inequality, for every $\epsilon>0$ there exists $C_\epsilon>0$ such that
  \[ \| \tilde A_{\tilde B} u \|_{\E_p^{(0)}(\Omega)} \le \epsilon \| A_{\tilde B} u\|_{\E_p^{(0)}(\Omega)} + C_\epsilon \|u\|_{\E_p^{(0)}(\Omega)} \quad (u\in D(A_{\tilde B})).\]
  Due to part (i) of the proof, $A_{\tilde B}$ is $\mathcal R$-sectorial, and by an abstract perturbation result on $\mathcal R$-sectorial operators (\cite{kunstmann-weis04}, Corollary~6.7), the same holds for $\tilde A_{\tilde B}$.
\end{proof}

\begin{remark}
  \label{3.4}
  Whereas the lower-order perturbation of the operator $A(D)$ could be handled by an abstract perturbation result on $\mathcal R$-boundedness, to our knowledge there is no such theorem on boundary perturbation which could be applied to our situation. Therefore, the proof of Lemma~\ref{3.3} directly uses the structure of the solution operators.
\end{remark}

The results above were formulated in a general setting in $\R^N$ with $N\ge 2$. In the physically relevant case $N=2$, the modelling can be found in \cite{lagnese89}, Chapter~2. Apart from physical constants, the equation in a uniform $C^4$-domain $\Omega\subset\R^2$ is given by
\begin{equation}
  \label{3-11}
  \begin{aligned}
    u_{tt}+\Delta^2 u + \Delta \theta & = 0\quad\text{ in } (0,\infty)\times\Omega,\\
    \theta_t - \Delta \theta - \Delta u_t & = 0\quad\text{ in }(0,\infty)\times \Omega
  \end{aligned}
\end{equation}
with boundary conditions
\begin{equation}
  \label{3-12}
  \begin{aligned}
    \Delta u + (1-\mu) B_1 u + \theta & = 0 \quad\text{ on }(0,\infty)\times \Gamma,\\
    \partial_\nu \Delta u + (1-\mu)B_2 u + \partial_\nu\theta & = 0 \quad\text{ on } (0,\infty)\times \Gamma,\\
    \partial_\nu\theta & = 0\quad\text{ on }(0,\infty)\times\Gamma.
  \end{aligned}
\end{equation}
Here, the operators $B_1$ and $B_2$ are given by
\begin{align*}
  B_1 u & := 2\nu_1\nu_2 u_{xy} -\nu_1^2 u_{yy} - \nu_2^2 u_{xx},\\
  B_2 u &:= \partial_\tau\big[ (\nu_1^2-\nu_2^2) u_{xy} + \nu_1\nu_2 (u_{yy}-u_{xx})\big],
\end{align*}
where $\nu = \binom{\nu_1}{\nu_2}$ denotes the outer normal vector and $\tau:=\binom{-\nu_2}{\nu_1}$. In \cite{lasiecka-triggiani98c}, the following variant of boundary conditions was considered:
\begin{equation}
  \label{3-13}
  \begin{aligned}
    \Delta u + (1-\mu) B_1 u + \theta & = 0 \quad\text{ on }(0,\infty)\times \Gamma,\\
    \partial_\nu \Delta u + (1-\mu)B_2 u -u + \partial_\nu\theta & = 0 \quad\text{ on } (0,\infty)\times \Gamma,\\
    \partial_\nu\theta + b\theta & = 0\quad\text{ on }(0,\infty)\times\Gamma
  \end{aligned}
\end{equation}
with $b>0$.

\begin{corollary}
  \label{3.5} Let $N=2$, and let $\Omega\subset\R^N$ be a uniform $C^4$-domain. Then the statements of Theorem~\ref{3.2} hold for the operators related to the boundary value problems \eqref{3-11}, \eqref{3-12} and  \eqref{3-11}, \eqref{3-13}.
\end{corollary}

\begin{proof}
  A straight-forward calculation shows that $B_1u = -\Delta'u $ and $B_2 u = \partial_\nu\Delta'u$ holds up to lower-order terms. Therefore, we can apply Lemma~\ref{3.3} to both boundary value problems.
\end{proof}

Finally, we study exponential stability in the case of a bounded domain.

\begin{theorem}
  \label{3.6}
  Let $\Omega\subset\R^N$, $N\ge 2$, be a bounded $C^4$-domain, and let $(T(t))_{t\ge 0}\subset L(\E_{p}^{(0)}(\Omega))$ be the $C_0$-semigroup generated by $A_{p,\Omega}$, see Theorem~\ref{3.2}. Let $P_{p,\Omega}\in L(\E_{p}^{(0)}(\Omega))$ denote the spectral projection corresponding to the eigenvalue $0$ of $A_{p,\Omega}$, and let $(T_0(t))_{t\ge 0}\subset L(\ker P_{p,\Omega})$ be the part of $T(t)$ in $\ker P_{p,\Omega}$, i.e., $T_0(t):= T(t)|_{\ker P_{p,\Omega}}$.

  Then $(T_0(t))_{t\ge 0}$ is exponentially stable, i.e., there exist $C>0$ and $\epsilon>0$ such that $\|T(t)\|_{L(\ker P_{p,\Omega})}\le C e^{-\epsilon t}\;(t\ge 0)$. The same holds for the perturbed problem $\tilde A_{p,\Omega}$ as in Lemma~\ref{3.3}.
\end{theorem}

\begin{proof}
  As $\Omega$ is bounded, the operator $A_{p,\Omega}$ has compact resolvent and discrete spectrum. Moreover, the spectrum is independent of $p\in (1,\infty)$. It was shown in \cite{lasiecka-triggiani98c} that $A_{2,\Omega}$ is dissipative which implies $\sigma(A_{2,\Omega})\subset \{\lambda\in\C: \Re\lambda\le 0\}$. Moreover, 0 is the only eigenvalue on the imaginary axis. Now the statements of the theorem follow from general semigroup theory.
\end{proof}

\begin{corollary}
  \label{3.7}
  Let $N=2$, and let $\Omega\subset\R^2$ be a bounded $C^4$-domain. Then the analytic semigroup related to the boundary value problem \eqref{3-11}, \eqref{3-12} is exponentially stable in the space $\ker P_{p,\Omega}$, and the analytic semigroup related to \eqref{3-11}, \eqref{3-13} is exponentially stable in the whole space $\E_p^{(0)}(\Omega)$.
\end{corollary}

\begin{proof}
  This is a particular case of Theorem~\ref{3.6} where we note that in the case of \eqref{3-11}, \eqref{3-13} there is no eigenvalue on the imaginary axis due to \cite{lasiecka-triggiani98c}.
\end{proof}


\begin{thebibliography}{10}

\bibitem{chuesov-lasiecka10}
I.~Chueshov and I.~Lasiecka.
\newblock {\em Von {K}arman evolution equations}.
\newblock Springer Monographs in Mathematics. Springer, New York, 2010.
\newblock Well-posedness and long-time dynamics.

\bibitem{denk-hieber-pruess03}
R.~Denk, M.~Hieber, and J.~Pr{\"u}ss.
\newblock {$\mathcal R$}-boundedness, {F}ourier multipliers and problems of
  elliptic and parabolic type.
\newblock {\em Mem. Amer. Math. Soc.}, 166(788):viii+114, 2003.

\bibitem{denk-racke06}
R.~Denk and R.~Racke.
\newblock {$L^p$}-resolvent estimates and time decay for generalized
  thermoelastic plate equations.
\newblock {\em Electron. J. Differential Equations}, pages No. 48, 16 pp.
  (electronic), 2006.

\bibitem{denk-schnaubelt15}
R.~Denk and R.~Schnaubelt.
\newblock A structurally damped plate equation with {D}irichlet--{N}eumann
  boundary conditions.
\newblock {\em J. Differential Equations}, 259(4):1323--1353, 2015.

\bibitem{Denk-Shibata17}
R.~Denk and Y.~Shibata.
\newblock Maximal regularity for the thermoelastic plate equations with free
  boundary conditions.
\newblock {\em J. Evol. Equ.}, 17(1):215--261, 2017.

\bibitem{enomoto-shibata13}
Y.~Enomoto and Y.~Shibata.
\newblock On the {$\mathcal R$}-sectoriality and the initial boundary value
  problem for the viscous compressible fluid flow.
\newblock {\em Funkcial. Ekvac.}, 56(3):441--505, 2013.

\bibitem{girardi-weis03}
M.~Girardi and L.~Weis.
\newblock Criteria for {R}-boundedness of operator families.
\newblock In {\em Evolution equations}, volume 234 of {\em Lecture Notes in
  Pure and Appl. Math.}, pages 203--221. Dekker, New York, 2003.

\bibitem{kim92}
J.~U. Kim.
\newblock On the energy decay of a linear thermoelastic bar and plate.
\newblock {\em SIAM J. Math. Anal.}, 23(4):889--899, 1992.

\bibitem{kunstmann-weis04}
P.~C. Kunstmann and L.~Weis.
\newblock Maximal {$L_p$}-regularity for parabolic equations, {F}ourier
  multiplier theorems and {$H^\infty$}-functional calculus.
\newblock In {\em Functional analytic methods for evolution equations}, volume
  1855 of {\em Lecture Notes in Math.}, pages 65--311. Springer, Berlin, 2004.

\bibitem{lagnese89}
J.~E. Lagnese.
\newblock {\em Boundary stabilization of thin plates}, volume~10 of {\em SIAM
  Studies in Applied Mathematics}.
\newblock Society for Industrial and Applied Mathematics (SIAM), Philadelphia,
  PA, 1989.

\bibitem{lasiecka-triggiani97}
I.~Lasiecka and R.~Triggiani.
\newblock Analyticity, and lack thereof, of thermo-elastic semigroups.
\newblock In {\em Control and partial differential equations
  ({M}arseille-{L}uminy, 1997)}, volume~4 of {\em ESAIM Proc.}, pages 199--222
  (electronic). Soc. Math. Appl. Indust., Paris, 1998.

\bibitem{lasiecka-triggiani98c}
I.~Lasiecka and R.~Triggiani.
\newblock Analyticity of thermo-elastic semigroups with free boundary
  conditions.
\newblock {\em Ann. Scuola Norm. Sup. Pisa Cl. Sci. (4)}, 27(3-4):457--482
  (1999), 1998.

\bibitem{lasiecka-wilke13}
I.~Lasiecka and M.~Wilke.
\newblock Maximal regularity and global existence of solutions to a quasilinear
  thermoelastic plate system.
\newblock {\em Discrete Contin. Dyn. Syst.}, 33(11-12):5189--5202, 2013.

\bibitem{liu-liu97}
K.~Liu and Z.~Liu.
\newblock Exponential stability and analyticity of abstract linear
  thermoelastic systems.
\newblock {\em Z. Angew. Math. Phys.}, 48(6):885--904, 1997.

\bibitem{liu-yong98}
Z.~Liu and J.~Yong.
\newblock Qualitative properties of certain {$C_0$} semigroups arising in
  elastic systems with various dampings.
\newblock {\em Adv. Differential Equations}, 3(5):643--686, 1998.

\bibitem{liu-zheng97}
Z.~Liu and S.~Zheng.
\newblock Exponential stability of the {K}irchhoff plate with thermal or
  viscoelastic damping.
\newblock {\em Quart. Appl. Math.}, 55(3):551--564, 1997.

\bibitem{liu-renardy95}
Z.-Y. Liu and M.~Renardy.
\newblock A note on the equations of a thermoelastic plate.
\newblock {\em Appl. Math. Lett.}, 8(3):1--6, 1995.

\bibitem{munos_rivera-racke95}
J.~E. Mu{\~n}oz~Rivera and R.~Racke.
\newblock Smoothing properties, decay, and global existence of solutions to
  nonlinear coupled systems of thermoelastic type.
\newblock {\em SIAM J. Math. Anal.}, 26(6):1547--1563, 1995.

\bibitem{naito09}
Y.~Naito.
\newblock On the {$L_p$}-{$L_q$} maximal regularity for the linear
  thermoelastic plate equation in a bounded domain.
\newblock {\em Math. Methods Appl. Sci.}, 32(13):1609--1637, 2009.

\bibitem{naito-shibata09}
Y.~Naito and Y.~Shibata.
\newblock On the {$L_p$} analytic semigroup associated with the linear
  thermoelastic plate equations in the half-space.
\newblock {\em J. Math. Soc. Japan}, 61(4):971--1011, 2009.

\bibitem{schade-shibata15}
K.~Schade and Y.~Shibata.
\newblock On strong dynamics of compressible nematic liquid crystals.
\newblock {\em SIAM J. Math. Anal.}, 47(5):3963--3992, 2015.

\bibitem{shibata94}
Y.~Shibata.
\newblock On the exponential decay of the energy of a linear thermoelastic
  plate.
\newblock {\em Mat. Apl. Comput.}, 13(2):81--102, 1994.

\bibitem{weis01}
L.~Weis.
\newblock Operator-valued {F}ourier multiplier theorems and maximal
  {$L_p$}-regularity.
\newblock {\em Math. Ann.}, 319(4):735--758, 2001.

\end{thebibliography}
\end{document}